\documentclass[11pt]{amsart}

\usepackage[bookmarks=true]{hyperref}

\usepackage{mathptmx}
\usepackage{amsmath}
\usepackage{amssymb}\begin{normalsize}\end{normalsize}
\usepackage{amsthm}
\usepackage{array}
\usepackage[all,tips]{xy}
\usepackage{enumerate}
\usepackage{graphicx}
\usepackage{lmodern}
\usepackage{mathrsfs}

\newtheorem{theorem}{Theorem}

\newtheorem{prop}[theorem]{Proposition}

\theoremstyle{definition}

\numberwithin{equation}{section}

\DeclareMathOperator{\Spin}{Spin}
\DeclareMathOperator{\SSpin}{\mathbf{Spin}}
\DeclareMathOperator{\SO}{SO}

\DeclareMathOperator{\Aut}{Aut}

\DeclareMathOperator{\Isom}{Isom}

\newcommand{\Z}{\mathbb Z}
\newcommand{\Q}{\mathbb Q}
\newcommand{\R}{\mathbb R}

\newcommand{\A}{\mathbb A}

\newcommand{\Af}{\A_\mathrm{f}}

\newcommand{\V}{V}
\newcommand{\Vf}{\V_{\mathrm{f}}}

\newcommand{\tA}{\mathrm{A}}

\newcommand{\tD}{\mathrm{D}}

\newcommand{\Hy}{\mathbf H}

\DeclareMathOperator{\vol}{vol}

\newcommand{\G}{\mathrm G}

\newcommand{\CG}{\mathrm C}

\newcommand{\D}{\mathscr D}
\newcommand{\AL}{\mathbf A}
\newcommand{\LL}{\mathbf L}

\newcommand{\bs}{\backslash}

\newcommand{\zero}{\xymatrix{ *={\bullet} \ar@{-}[r]^5&
				*={\bullet} \ar@{-}[r] & *={\bullet} \ar@{-}[r] & *={\bullet}
				\ar@{-}[r]  & *={\bullet} \ar@{-}[r] &
				*={\bullet} \ar@{--}[r] & *={\bullet}
			}}

\newcommand{\one}{\xymatrix@R=5pt@C=17.5pt{& & & & &  *={\bullet} \ar@{--}[dd] \\
				*={\bullet} \ar@{-}[r]^5 & *={\bullet} \ar@{-}[r] & *={\bullet}
				\ar@{-}[r]  & *={\bullet} \ar@{-}[r] &
				*={\bullet} \ar@{-}[ur] \ar@{-}[dr] & \\
				 & & & & & *={\bullet} 
			}}

\newcommand{\two}{\xymatrix{ *={\bullet} \ar@{-}[r]^5&
				*={\bullet} \ar@{-}[r] & *={\bullet} \ar@{-}[r] & *={\bullet}
				\ar@{-}[r]  & *={\bullet} \ar@{-}[r]^4 &
				*={\bullet} \ar@{--}[r] & *={\bullet}
			}}

\begin{document}
\title{The three smallest compact arithmetic hyperbolic $5$-orbifolds}

\author{Vincent Emery and Ruth Kellerhals}
\thanks{Kellerhals partially supported by the Swiss National Science Foundation, project no. 200020-131967}
\address{
Max Planck Institute for Mathematics\\
Vivatsgasse 7\\
53111 Bonn\\
Germany\\
\newline
Department of Mathematics\\
University of Fribourg\\
Chemin du Mus\'ee 23\\
1700 Fribourg\\
Switzerland
}
\email{vincent.emery@gmail.com
\\
ruth.kellerhals@unifr.ch}

\date{\today}


\maketitle

\section{Introduction}

Let $\Isom(\Hy^5)$ be the group of isometries of
the hyperbolic space $\Hy^5$ of dimension five, and $\Isom^+(\Hy^5)$
its index two subgroup of orientation-preserving isometries.
In~\cite{BelEme} (see also~\cite{EmePhD}) the lattice of
smallest covolume among cocompact arithmetic lattices of $\Isom^+(\Hy^5)$
was determined. This lattice was constructed as
the image of an arithmetic subgroup $\Gamma_0$ of the spinor group $\Spin(1,5)$ (note
that $\Spin(1,n)$ is a twofold covering of $\SO(1,n)^\circ \cong
\Isom^+(\Hy^n)$).
More precisely, $\Gamma_0$ is given by the normalizer
in $\Spin(1,5)$ of a certain arithmetic group $\Lambda_0
\subset \G_0(k_0)$, where $k_0=\Q(\sqrt{5})$ and $\G_0$ is the algebraic
$k_0$-group $\SSpin(f_0)$ defined by the quadratic form
\begin{align}
	f_0(x) &=  - (3+2\sqrt{5}) x_0^2 + x_1^2 + \cdots + x_5^2	\, .
	\label{eq:def-f0}
\end{align}
In~\cite{BelEme} the index $[\Gamma_0: \Lambda_0]$ was computed to be
equal to $2$. We note that it is easily checked that $\Lambda_0$
contains the center of $\Spin(1,5)$, so that the covolume of the action
of $\Lambda_0$ on $\Hy^5$ is the double of the covolume of $\Gamma_0$.

In this article we construct a cocompact arithmetic lattice $\Gamma_2
\subset \Spin(1,5)$ of covolume slightly bigger than the covolume of
$\Lambda_0$, and we prove that it realizes the third smallest covolume
among cocompact arithmetic lattices in $\Spin(1,5)$. In other words, we
obtain the second and third values in the volume spectrum of compact
orientable arithmetic hyperbolic $5$-orbifolds, thus improving the
results of~\cite{BelEme,EmePhD} for this dimension. For notational reasons
we put $\Gamma_1 = \Lambda_0$. Moreover, for $i= 0,1,2$,
we denote by $\Gamma_i'$ the image of $\Gamma_i$ in $\Isom^+(\Hy^5)$.

\begin{theorem}
	\label{thm:vol}
	The lattices $\Gamma'_0, \Gamma'_1$ and  $\Gamma'_2$ (ordered by increasing covolume)
	are the three cocompact arithmetic lattices in $\Isom^+(\Hy^5)$ of minimal
	covolume. They are unique, in the sense that any cocompact
	arithmetic lattice in $\Isom^+(\Hy^5)$ of covolume smaller than or
	equal to $\Gamma_2'$ is conjugate in $\Isom(\Hy^5)$ to
	one of the $\Gamma_i'$.
\end{theorem}

The precise formulas for the hyperbolic covolumes of these lattices are
given below in Proposition \ref{prop:index-2}. We list in Table~\ref{tab:covol} the corresponding numerical
approximations.

\begin{table}
	\centering
	\begin{tabular}{cr}
		Lattice &  Hyperbolic covolume \\[3pt]
		\hline\\[-5pt]
		$\Gamma'_0$  & 0.00153459236\dots \\[3pt]
		$\Gamma'_1$  & 0.00306918472\dots \\[3pt]
		$\Gamma'_2$  & 0.00396939286\dots \\[5pt]
	\end{tabular}
	\end{table}
\begin{table}
	\centering
	\begin{tabular}{clr}
		Coxeter group & Coxeter symbol & Hyperbolic covolume \\[3pt]
		\hline\\[-5pt]
		$\Delta_0$  & $[5,3,3,3,3]$ & $0.00076729618\dots$ \\[3pt]
	         $\Delta_1$  & $[5,3,3,3,3^{1,1}]$ & $0.00153459235\dots$ \\[3pt]
		$\Delta_2$  & $[5,3,3,3,4]$ & $0.00198469643\dots$ \\[10pt]
	\end{tabular}		
	\caption{Approximation of hyperbolic covolumes}
	\label{tab:covol}
\end{table}

A central motivation for Theorem \ref{thm:vol} is that the lattices
$\Gamma'_0$, $\Gamma'_1$ and $\Gamma'_2$  can be
related to concrete geometric objects. Namely, let $P_0$ and $P_2$ be the two
compact Coxeter polytopes in $\Hy^5$ described by the following Coxeter
diagrams, of respective Coxeter symbols $[5,3,3,3,3]$ and $[5,3,3,3,4]$
(see \S 4):
\begin{align}
	P_0 &: \quad \zero
	\label{eq:Cox-0}
\end{align}
\begin{align}
	P_2 &: \quad \two
	\label{eq:Cox-2}
\end{align}
These two polytopes were first discovered by Makarov \cite{Mak68} (see
also Im Hof \cite{ImH90}) (see  \S\ref{sec:proof-2}). Combinatorially, they are simplicial prisms. Let $P_1=DP_0$ be the geometric double of $P_0$ with respect to its Coxeter facet
$\,[5,3,3,3]$. It follows that the Coxeter polytope $P_1$ can be
characterized by the following Coxeter diagram, of symbol $[5,3,3,3,3^{1,1}]$:
\begin{align}
	P_1 &: \quad \vcenter{\one}
	\label{eq:Cox-1}
\end{align}
We denote by $\Delta_i \subset \Isom(\Hy^5)$ the Coxeter
group generated by the reflections through the hyperplanes delimiting
$P_i$ ($0\le i\le 2$). It is known, by Vinberg's criterion \cite{Vinb67},
that the lattices $\Delta_0$ (thus $\Delta_1$ as well) and $\Delta_2$
are arithmetic.
\begin{theorem}
	For $i=0,1,2,$ let $\Delta_i^+$ be the lattice $\Delta_i \cap
	\Isom^+(\Hy^5)$, which is of index two in $\Delta_i$. Then
	$\Delta^+_i$ is conjugate to $\Gamma_i'$ in $\Isom(\Hy^5)$.
	In particular, $\Delta_0$ realizes the smallest covolume 
	among the cocompact arithmetic lattices in $\Isom(\Hy^5)$.
	\label{thm:Cox}
\end{theorem}
The proof of Theorem \ref{thm:Cox} is obtained as a consequence of
Theorem~\ref{thm:vol} (more exactly from the slightly more precise
Proposition \ref{prop:smallest-three}) together with an geometric/analytic
computation of the volumes $\vol(P_0)$ and $\vol(P_2)$ that will be
presented in
\S\ref{sec:proof-2}. We note that the fact that $\Delta_2$ and
$\Gamma_2'$ are commensurable lattices follows from the work of Bugaenko
\cite{Buga84} where $\Delta_2$ is constructed by applying Vinberg's algorithm
on the same quadratic form~\eqref{eq:def-f2} which we will use below to
construct $\Gamma_2$. No arithmetic construction of $\Delta_0$ and
$\Delta_1$ was known so far.
 
The approximations of the volumes of $P_0$, $P_1$ and $P_2$ are listed in
Table~\ref{tab:covol}. These volumes can be obtained by two completely
different approaches: from the method given in \S \ref{sec:proof-2}, or
from the covolumes of the arithmetic lattices $\Gamma_i$, which are
essentially computed with
Prasad's volume formula \cite{Pra89}.  The comparison of these two
approaches has some arithmetic significance that will be briefly
discussed in \S \ref{sec:rmks-vol}.

\subsection*{Acknowledgements}

We would like to thank Herbert Gangl for interesting discussions
concerning \S \ref{sec:rmks-vol}. We thank the Institut Mittag-Leffler
in Stockholm, where this paper was completed. The first named author is
thankful to the MPIM in Bonn for the hospitality and the financial
support.

\section{Construction and properties of $\Gamma_2$}
\label{sec:Gamma2}

We call an algebraic group \emph{admissible} if it gives rise to
cocompact lattices in $\Spin(1,5)$; see \cite[\S 2.2]{BelEme} for the exact
definition. We say that an admissible $k$-group $\G$ is
\emph{associated with} $k/\ell$, where $\ell$ is the smallest field extension
of $k$ (necessarily quadratic) such that $\G$ is an inner form over
$\ell$, sometimes called ``splitting field'' of $\G$. We use the same
terminology for the arithmetic subgroups of $\G$.
Admissibility imposes that $\G$ is of type $^2\tA_3$, the field
$k$ is totally real, and $\ell$ has signature $(2,d-1)$ where $d =
[k:\Q]$ (cf. \cite[Prop. 2.5]{BelEme}). Note that since we consider only
cocompact lattices in this article, we assume that $k \neq \Q$.
In the following, the symbol $\Vf$ will always refer to the set of
finite places of the base field $k$ (and not of $\ell$).

Let $\G_2$ be the algebraic spinor group $\SSpin(f_2)$ defined over $k_0
= \Q(\sqrt{5})$, where $f_2$ is the following quadratic form:
\begin{align}
	f_2(x) &=   - \omega x_0^2 + x_1^2 + \cdots + x_5^2	\, ,
	\label{eq:def-f2}
\end{align}
with $\omega = \frac{1 + \sqrt{5}}{2}$. We have $\G_2(\R) \cong
\Spin(1,5)
\times \Spin(6)$, proving that $\G_2$ is admissible. Its ``splitting
field'' is given by (cf. \cite[\S 3.2]{BelEme}):
\begin{align}
	\label{eq:ell1}
	\ell_2 &= \Q(\sqrt{\omega})\\
	\nonumber	&\cong \Q[x]/(x^4 - x^2 -1) \, ,
\end{align}
which has a discriminant of absolute value $\D_{\ell_2} = 400$.
The following proposition shows
an analogy between $\G_2$ and $\G_0$ (cf. \cite[Prop. 3.6]{BelEme}).

\begin{prop}
	\label{prop:G2-qs}
	The group $\G_2$ is quasisplit at every finite
	place $v$ of $k_0$. It is the unique admissible group associated
	with $k_0/\ell_2$ with this property.
\end{prop}
\begin{proof}
	Since $\omega$ is an integer unit in $k_0$ it easily
	follows that for at each nondyadic place $v \neq (2)$ the form $f_2$ has the
	same Hasse symbol as the standard split form 
	of signature $(3,3)$. From the structure theory of
	$\SSpin$ described in \cite[\S 3.2]{BelEme} we conclude that
	$\G_2$ must be quasisplit at every finite place $v$ (note that at the
	place $v = (2)$, which is ramified in $\ell_2/k_0$, the group $\G_2$
	is necessarily an outer form). Similarly to the proof of
	\cite[Prop. 3.6]{BelEme}, the second affirmation follows from
	\cite[Lemma 3.4]{BelEme} together with the Hasse-Minkowski theorem.
\end{proof}

We write here $k = k_0$. By Proposition \ref{prop:G2-qs} we see that for every finite place $v
\in \Vf$ there exists a special parahoric subgroup $P_v \subset \G_2(k_v)$. 
More precisely, $P_v$ is hyperspecial unless $v$ is the dyadic place $(2)$ (the particularity of
$v=(2)$ comes from the fact that this place is ramified in the extension
$\ell_2/k_0$).
The collection $(P_v)_{v \in \Vf}$ of special parahoric
subgroups can be chosen to be coherent, i.e., such that $\prod_{v} P_v$
is open in the group $\G_2(\Af)$ of finite adelic points. We now consider the principal
arithmetic subgroup associated with such a coherent collection:
\begin{align}
	\Lambda_2 &= \G_2(k_0) \cap \prod_{v \in \Vf} P_v \, .
	\label{eq:def-Lbda2}
\end{align}
The covolume of $\Lambda_2$ can be computed with Prasad's volume
formula \cite{Pra89}. If $\mu$ denotes the Haar measure on $\Spin(1,5)$
normalized as in \cite{BelEme} (which corresponds to the measure
$\mu_\mathrm{S}$ in \cite{Pra89}), then we obtain:
\begin{align}
	\mu(\Lambda_2 \bs \Spin(1,5) ) &= 	\D_{k_0}^{15/2}
	\D_{\ell_2}^{5/2} C^2 \zeta_{k_0}(2) \zeta_{k_0}(4) L_{\ell_2/k_0}(3) 
	\, , \label{eq:mu-Lbda2}
\end{align}
where $C = 3\cdot2^{-7} \pi^{-9}$, the symbol $\zeta_k$ denotes the Dedekind zeta function
associated with $k$, and $L_{\ell/k} = \zeta_{\ell} /
\zeta_{k}$ is the $L$-function corresponding to a quadratic extension
$\ell/k$.

We can now construct the group $\Gamma_2$ and compute its
hyperbolic covolume. In the same proposition we recall
the value of the hyperbolic covolume of $\Gamma_0$, which was obtained
in \cite{BelEme}.

\begin{prop}
	\label{prop:index-2}
	Let $\Gamma_2$ be the normalizer of $\Lambda_2$ in $\Spin(1,5)$.
	Then $\Lambda_2$ has index two in $\Gamma_2$. It follows that
	the hyperbolic covolume of $\Gamma'_2$ is equal to 
	\begin{align}
		\frac{9 \sqrt{5}^{15}}{2^3 \pi^{15}}
	\zeta_{k_0}(2) \zeta_{k_0}(4) L_{\ell_2/k_0}(3)
	&= 0.00396939286\dots\, .
		\label{eq:vol-Gam2}
	\end{align}
	The hyperbolic covolume of $\Gamma'_0$ is equal to 	
	\begin{align}
		\frac{9 \sqrt{5}^{15} \sqrt{11}^5}{2^{14} \pi^{15}}
	\zeta_{k_0}(2) \zeta_{k_0}(4) L_{\ell_0/k_0}(3)
	&= 0.00153459236\dots\, ,
		\label{eq:vol-Gam0}
	\end{align}
	where $\ell_0$ is the quartic field with $x^4 - x^3 + 3x -1$ as
	defining polynomial.
\end{prop}
\begin{proof}
	The relation between the measure $\mu$ and the hyperbolic volume
	is described in \cite[\S 2.1]{BelEme}, where it is proved that
	in dimension $5$ the hyperbolic covolume corresponds to the
	covolume with respect to $2 \pi^3 \times \mu$. Thus 
	it remains to prove that $[\Gamma_2:\Lambda_2] = 2$.
	Let $k=k_0$.

	It follows from the theory developed in \cite[\S 4]{BelEme}
	that the index $[\Gamma_2:\Lambda_2]$ is equal to the order of
	the group denoted by $A_\xi$ in {\em loc.\,cit.}, which can be
	identified as a subgroup of index at most two in $\AL_4/(\ell_2^\times)^4$, where 
	\begin{align}
		\AL &= \left\{ \left.  x \in \ell_2^\times \; \right| \; N_{\ell_2/k}(x) \in
		(k^\times)^4 \mbox{ and } x>0  \right\} \; ;\\
		\label{eq:def-AL}
		\AL_4 &= \left\{ \left. x \in \AL \; \right| \; 
		\nu(x) \in 4\Z \mbox{ for each normalized valuation }\nu
		\mbox{ of } \ell_2 \right\}.
	\end{align}
	Note that in particular, for the integers $q$ and $q'$
	introduced in \cite[\S 4.9]{BelEme}, we have $q = \overline{q} =
	1$. Moreover, if $v=(2)$ denotes the (unique) ramified place of $\ell_2/k$,
	the subgroup $A_\xi$ is proper of index two in $\AL_4/(\ell_2^\times)^4$ if
	and only if there exists an element of $\AL_4$ acting nontrivially
	on the local Dynkin diagram $\Delta_v$ of $\G_2( k_v)$.
	The action of $\AL$ on $\Delta_v$ comes from
	its identification as a subgroup of the first Galois cohomology
	group $H^1(k, \CG)$ (where $\CG$ is the center of $\G_2$),
	which acts on every local Dynkin diagram associated
	with $\G_2$. Since $\G_2$ is of type $\tA$, we can use the
	results of \cite[\S 4.2]{MohamSG}, which show that if
	$\pi_w \in \ell_2$ is a uniformizer for the ramified place $w|v$
	of $\ell_2$, then $s = \pi_w \overline{\pi_w}^{-1}$ is a
	generator of the group $\Aut(\Delta_v)$. Taking $\pi_w
	= 1 + \omega + \sqrt{\omega}$, we obtain a positive unit $s$
	acting nontrivially on $\Delta_v$. Thus, $A_\xi$ has index
	two in $\AL_4/(\ell_2^\times)^4$. But the order of this latter
	group was computed in \cite[\S 7.5]{BelEme} to be equal to $4$.
	This gives $[\Gamma_2:\Lambda_2] = 2$.
\end{proof}

The ``uniqueness'' part of Theorem~\ref{thm:vol} requires the following
result.
\begin{prop}
	\label{prop:uniq-G2}
	Up to conjugacy, the image of $\Gamma_2$ in $\Isom(\Hy^5)$ does
	not depend on the choice of a coherent collection of
	special parahoric subgroups $P_v \subset \G_2(k_v)$.
\end{prop}
\begin{proof}
	To prove this we can follow the same line of arguments as in
	\cite[\S 6]{BelEme}, where the result is proved for $\Gamma_0
	\subset \G_0$ (our situation corresponding to the case of the type
	$^2\tD_{2m+1}$). Thus, using \cite[\S 6.5]{BelEme}, the result 
	follows by checking that $\LL/\AL$ and $U_\LL/U_\AL$ have the
	same order (equal to 2), where 
	\begin{align}
		\LL &= \left\{ \left.  x \in \ell_2^\times \; \right| \;
		N_{\ell_2/k_0}(x) \in (k_0^\times)^4 \right\}
		\label{def-LL}
	\end{align}
	and $U_\LL$ (resp. $U_\AL$) is the intersection of $\LL$ (resp.
	$\AL$) with the integers units in $\ell_2$.
\end{proof}

\section{Proof of Theorem \ref{thm:vol}}
\label{sec:proof-1}

In view of Proposition~\ref{prop:index-2}, Theorem~\ref{thm:vol} is a
direct consequence of the following statement.

\begin{prop}
	\label{prop:smallest-three}
	Let $\Gamma' \subset \Isom^+(\Hy^5)$ be a cocompact arithmetic lattice that is
	not conjugate to $\Gamma'_0$, $\Gamma'_1$ or $\Gamma'_2$. Then
	$\vol(\Gamma'\bs \Hy^5) > 4 \cdot 10^{-3}$.
\end{prop}
\begin{proof}
Let $\Gamma \subset \Spin(1,5)$ be the full inverse image of $\Gamma'$.
We suppose that $\Gamma$ is an arithmetic subgroup of the group
$\G$ defined, associated with  $\ell/k$.
From the values given in \eqref{eq:vol-Gam2} and \eqref{eq:vol-Gam0}, it
is clear that if $\Gamma$ is a proper subgroup of $\Gamma_0$,
$\Gamma_1$ or $\Gamma_2$, then $\vol(\Gamma' \bs \Hy^5) > 4 \cdot
10^{-3}$. Thus it suffices to prove the result assuming that $\Gamma$ is a maximal arithmetic
subgroup with respect to inclusion. In particular, $\Gamma$ can be
written as the normalizer of the principal arithmetic subgroup $\Lambda$
associated with some coherent collection $P = (P_v)$ of parahoric
subgroups $P_v \subset \G(k_v)$.

First we suppose that $k = k_0$, and $\ell = \ell_0$ or $\ell_2$. By
Proposition \ref{prop:G2-qs} and its analogue for $\G_0$, if $\G$ is not
isomorphic to $\G_0$ or $\G_2$ then at least one $P_v$ is not special.
In particular, a ``lambda factor'' $\lambda_v \ge 18$ appears in the
volume formula of $\Lambda$ \cite[\S 7.1]{BelEme}. Together with
\cite[(15)]{BelEme} (note that we do not assume here that $\Gamma =
\Gamma^\mathfrak{m}$, in the notation of {\em loc.\,cit.}) this shows that the covolume of $\Gamma$ is at least
$9$ times the covolume of $\Gamma_0$. Now if $\G$ is isomorphic to
$\G_0$ or $\G_2$, Proposition \ref{prop:uniq-G2} and its analogue for
$\G_0$ show that at least one $P_v$ is not special, and the same
argument as above applies. 

Now we consider the situation $(k,\ell) \neq (k_0,\ell_0)$ nor
$(k_0,\ell_2)$. We will use the different lower bounds for the covolume
of $\Gamma$ given in \cite[\S 7]{BelEme}. Note that in our case the rank
$r$ of $\G$ is equal to $3$. The notations are the following: $d$ is the degree
of $k$, $\D_k$ and $\D_\ell$ are the discriminants of $k$ and $\ell$ in
absolute value, and $h_\ell$ is the class number of $\ell$. Moreover, we
set $a = 3^3 2^{-4} \pi^{-11}$.  From \cite[(37)]{BelEme} we have for
$d \ge 7$ the following lower bound, which proves the result in this
case (recall that the hyperbolic volume corresponds to $2\pi^3 \times \mu$,
where $\mu$ is the Haar measure used by Prasad).
\begin{align}
	\vol(\Gamma'\bs \Hy^5) &> \frac{2 \pi^3}{32} \left( 9.3^{5.5}
	\cdot a \right)^7 \label{eq:37} \\
	&= 7.657\ldots \nonumber
\end{align}

The following bound corresponds to \cite[(35)]{BelEme}.
\begin{align}
	\vol(\Gamma'\bs \Hy^5) &> \frac{2 \pi^3}{32} \D_k^{5.5} a^d
	\label{eq:35} 
\end{align}
For each degree $d=2, \dots, 6$ we can use \eqref{eq:35} to prove
the result for a discriminant $\D_k$ high enough (e.g., $\D_k \ge 27$
for $d=2$). This leave us with a finite number of possible fields $k$ to
examine. From these bounds on $\D_k$ and the tables of number fields
(such as \cite{Bordeaux_data} and \cite{Qaos}) we obtain a list of nineteen fields $k$ (none of
degree $d=6$) that remain to check.

Let us further consider the two following bounds, corresponding to
\cite[(34) and (31)]{BelEme}. See \eqref{eq:mu-Lbda2} for the value of
the symbol $C$.
\begin{align}
	\vol(\Gamma'\bs \Hy^5) &> \frac{2 \pi^3}{32} \D_k^{2.5}
	\D_\ell^{1.5} a^d\;;
	\label{eq:34} \\
	\vol(\Gamma'\bs \Hy^5) &> \frac{2\pi^3}{h_\ell 2^{d+1}} \D_k^{7.5}
	(\D_\ell/\D_k^2)^{2.5} C^d \;.
	\label{eq:31} 
\end{align}
For each of the nineteen fields $k$ we easily obtain an upper bound
$b_k$ for $\D_\ell$ for which the right hand side of \eqref{eq:34} is at most
$4 \cdot 10^{-3}$. Thus we only need to analyse the fields $\ell$ with
$\D_\ell \le b_k$. Let us fix a field $k$. The computational method
described \cite{CohDiaOl98}, based on class field theory, allows to
determine all the quadratic extensions $\ell/k$ with $\D_\ell \le b_k$
and with $\ell$ of right signature, that is, $(2,d-1)$ (cf. \cite[Prop.
2.5]{BelEme}). More precisely, we obtain this list of $\ell/k$ by programming a
procedure in Pari/GP that uses the built-in functions \texttt{bnrinit}
and \texttt{rnfkummer}. For each pair $(k,\ell)$ obtained,
\textsc{Pari/GP} gives us the class
number $h_\ell$ (checking its correctness with \texttt{bnfcertify}) and
this information makes \eqref{eq:31} usable. The inequality 
$\vol(\Gamma'\bs\Hy^5) > 4 \cdot 10^{-3}$ follows then for all the
remaining $(k,\ell)$ except for  the two situations:
\begin{align}
	(\D_k,\D_\ell) &= (8, 448)\; , \label{eq:l-448} \\
	(\D_k,\D_\ell) &= (5, 475)\;.  \label{eq:l-475} 
\end{align}

The case \eqref{eq:l-475} follows from Proposition \ref{prop:cas-475} below.
Let us then consider the case associated with \eqref{eq:l-448}. The
smallest possible covolume of a maximal arithmetic subgroup
$\Gamma = N_{\Spin(1,5)}(\Lambda)$ associated with $\ell/k$ would be in
the situation when all parahoric subgroup $P_v$ determining $\Lambda$
are special. In this case, by \cite[Prop. 4.12]{BelEme} the index
$[\Gamma:\Lambda]$ is bounded by $8$, and together with the precise covolume of
$\Lambda$ by Prasad's formula, we obtain (using \textsc{Pari/GP} to evaluate the zeta
functions):
\begin{align}
	\vol(\Gamma'\bs \Hy^5) &\ge \frac{2\pi^3}{8} \D_k^{7.5}
	(\D_\ell/\D_k^2)^{2.5} C^2  \zeta_k(2) \zeta_k(4)
	\zeta_\ell(3)/\zeta_k(3)  \label{eq:vol-448} \\
	&= 0.004997\dots \nonumber
\end{align}
This concludes the proof.
\end{proof}

\begin{prop}
	\label{prop:cas-475}
	Let $\ell$ be the quadratic extension of $k_0 = \Q(\sqrt{5})$ with discriminant
	of absolute value $\D_\ell = 475$. There exists a cocompact
	arithmetic lattice in $\Isom^+(\Hy^5)$ associated with $\ell/k_0$
	whose approximate hyperbolic covolume is $0.006094\dots$. This is
	the smallest covolume among arithmetic lattices in
	$\Isom^+(\Hy^5)$ associated with $\ell/k_0$.
\end{prop}
\begin{proof}
	Let $k=k_0$.
	The field $\ell$ can be concretely described as $\ell =
	k(\sqrt{\beta})$, where $\beta = -1 + 2\sqrt{5}$ (this is a
	divisor of 19).	We consider the algebraic group $\G =
	\mathbf{Spin}(f)$ defined over $k=k_0$, with 
	\begin{align}
		f(x) &=  - \beta x_0^2 + x_1^2 + \cdots + x_5^2	\, .
	\label{eq:f-475}
	\end{align}
	Similarly  to \cite[Prop. 3.6]{BelEme}, we have that $\G$ is
	quasisplit at every finite place $v \in \Vf$ (note that the
	proof for the unique dyadic place can be simplified in {\em
	loc.\,cit.} by noting $2$ is inert in $\ell$ and thus, $\G$
	must be an outer form, necessarily quasisplit, cf. \cite[\S
	3.2]{BelEme}). It follows that there exist a coherent collection
	of special parahoric subgroups $P_v \subset \G(k_v)$, and by
	Prasad's volume formula the hyperbolic covolume of an associated
	principal arithmetic subgroup $\Lambda$ is given by
	\begin{align}
	\vol(\Lambda\bs \Hy^5) &= 2\pi^3 \D_k^{7.5}
		(\D_\ell/\D_k^2)^{2.5} C^2  \zeta_k(2) \zeta_k(4)
		\zeta_\ell(3)/\zeta_k(3)  \label{eq:vol-475} 
	\end{align}
	The index $[\Gamma:\Lambda]$ of $\Lambda$ in its normalizer
	$\Gamma$ can be computed using the same method as in the proof
	of Proposition \ref{prop:index-2}. That the group
	$\AL_4/(\ell^\times)^4$ has order $4$ was already computed in
	\cite[\S 7.5]{BelEme}. We use again \cite[\S 4.2]{MohamSG} to analyse
	the behaviour at the ramified place $v = (\beta)$: for the
	uniformizer $\pi_w = \frac{\sqrt{\beta} + \beta}{2}$ of the
	place $w | v$ we get that $s = \pi_w \overline{\pi_w}^{-1}$ is
	an element of $\AL_4$ that acts nontrivially on the local Dynkin
	diagram $\Delta_v$ of $\G(k_v)$. As in the proof of Proposition
	\ref{prop:index-2} it follows that $[\Gamma:\Lambda] = 2$. From
	\eqref{eq:vol-475} we obtain the value $0.006094\dots$ as the
	hyperbolic covolume of $\Gamma$. That no other arithmetic group
	associated with $\ell/k$ has smaller covolume follows from
	\cite[\S 4.3]{BelEme} (since $\Lambda$ is of the form
	$\Lambda^\mathfrak{m}$; cf. \cite[\S 12.3]{EmePhD} for more
	details).

\end{proof}

\section{Proof of Theorem \ref{thm:Cox}}
\label{sec:proof-2}

Consider the vector space model $\mathbb R^{1,5}$ for $\Hy^5$ as above and represent 
a hyperbolic hyperplane $H=e^{\perp}$ 
by means of a space-like unit vector
$e\in \mathbb R^{1,5}$. A hyperbolic Coxeter polytope $P=\cap_{i\in I}H_i^{-}$ 
is the intersection of finitely many half-spaces (whose normal unit vectors are
directed outwards w.r.t. $P$ and) whose dihedral angles are submultiples of $\pi$.
The group $\Delta$ generated by the reflections with respect to the hyperplanes
$H_i, i\in I,$ is a discrete subgroup of $\Isom(\Hy^5)$.
If the cardinality of $I$ is small, a Coxeter polytope and its reflection group are best represented by
the \emph{Coxeter diagram} or by the \emph{Coxeter symbol}. 
To each limiting hyperplane $H_i$ of a Coxeter polytope $P$ corresponds a node $i$ in the
Coxeter diagram, and two nodes $i,j$ are connected by an edge of weight $p$ if 
the hyperplanes intersect under the (non-right) angle $\pi/p$. Notice that the weight $3$ will always be
omitted.
If two hyperplanes are orthogonal, their nodes are not connected. If they admit a common perpendicular
(of length $l$), their nodes are joined by a dashed edge (and the weight $l$ is usually omitted).
We extend the diagram description to arbitrary convex hyperbolic polytopes and 
associate with the dihedral angle $\alpha=\angle(H_i,H_j)$ an edge with weight
$\alpha$ connecting the nodes $i,j$. For the intermediate case
of \emph{quasi-Coxeter polytopes} whose dihedral
angles are rational multiples $p\pi/q$ of $\pi$, the edge weight will be $q/p$.
The Coxeter symbol is a bracketed expression encoding the form of the Coxeter
diagram in an abbreviated way. For example, $[p,q,r]$ is associated with a linear Coxeter diagram
with 3 edges of consecutive markings $p,q,r$. The Coxeter symbol $[3^{i,j,k}]$
denotes a group with Y-shaped Coxeter diagram
with strings of $i$, $j$ and $k$ edges emanating from a common node. However,
dashed edges are omitted leaving a connected graph. The Coxeter symbol can be extended
to the quasi-Coxeter case in an obvious way as well. 

We are particularly interested 
in the quasi-Coxeter groups $\Delta_i$ and the polytopes $P_i$ (see \S
1) as given in Table \ref{tab:examples}.
\begin{table}
	\centering
	\begin{tabular}{clcc}
		Group & Coxeter diagram & Coxeter symbol & Polytope \\[3pt]
		\hline\\[-5pt]
		$\Delta_0$  & $\zero$ & $[5,3,3,3,3]$ & $P_0$ \\[3pt]
		$\Delta_1$  & \hspace{-7pt} $\vcenter{\one}$ & $[5,3,3,3,3^{1,1}]$ & $P_1=DP_0$ \\[3pt]
		$\Delta_2$  & $\two$ & $[5,3,3,3,4]$ & $P_2$ \\[10pt]
	\end{tabular}		
	\caption{Three hyperbolic Coxeter groups and their 5-polytopes}
	\label{tab:examples}
\end{table}
In order to compute the volumes of $P_i$, we consider the 1-parameter sequence of compact 5-prisms with symbol
\begin{align}
	P(\alpha) &: \quad	[5,3,3,3,\alpha]
	\label{eq:prisms}
\end{align}
where $\alpha\in[\pi/4,2\pi/5]$. Geometrically, they are compactifications of 5-dimen\-sional orthoschemes
by cutting away the ultra-ideal principal vertices by the associated polar hyperplanes.
The sequence (\ref{eq:prisms}) contains the Coxeter polytopes $P_0=[5,3,3,3,3]$ and $P_2=[5,3,3,3,4]$
as well as the pseudo-Coxeter prism $[5,3,3,3,5/2]$. There is no closed volume formula 
for such polytopes known in terms of the dihedral angles. However, for certain non-compact limiting cases
and by means of scissors congruence techniques, exact volume expressions
could be derived \cite[\S 4.2]{Kel12}.
For example, 

\newfam\cyrfam
\font\tencyr=wncyr10
\def\cyr{\fam\cyrfam\tencyr}
\def\russianL{\mathop{\hbox{\cyr L}}}
\begin{align}
	\hbox{vol}_5([{5}/{2},3,3,5,{5}/{2}])=\frac{13\zeta(3)}{9600}+\frac{11}{1152}{\russianL}_3(\frac{\pi}{5})\,\,,	
	\label{eq:Cq}
\end{align}	
\begin{align}
	\hbox{vol}_5([5,3,3,{5}/{2},5])=-\frac{\zeta(3)}{4800}+\frac{11}{1152}{\russianL}_3(\frac{\pi}{5})\,\,,
	\label{eq:Cq2}
\end{align}
and finally,
\begin{align}
	\hbox{vol}_5(P({2\pi}/{5}))=\frac{1}{5}\left(\hbox{vol}_5([{5}/{2},3,3,5,{5}/{2}])-\hbox{vol}_5([5,3,3,{5}/{2},5])\right)=\frac{\zeta(3)}{3200}\,\,.
	\label{eq:prism0}
\end{align}
Here, 
\begin{align}
{\russianL}_3(\omega)=
{1\over4}\,\sum_{r=1}^{\infty}{\cos(2r\omega)\over r^3}
=\frac{1}{4}\,\zeta(3)-\int\limits_0^{\omega}\,{\russianL}_2(t) dt\,\,,\,\,\omega\in\mathbb R\,\,,
	\label{eq:loba3}
\end{align}
denotes the Lobachevsky trilogarithm function which is related to the real part of the classical polylogarithm 
$\hbox{Li}_k(z)=\sum_{r=1}^{\infty}z^r/r^k$  for $k=3$ and
$z=\exp(2i\omega)$ (see \cite[\S 4.1]{Kel12} and (\ref{eq:loba2})).

For the volume calculation of the prisms
$P_0$ and $P_2$, we apply the volume differential formula of L. Schl\"afli (see \cite{Mil}, for example)
with the reference value (\ref{eq:prism0}) in order to obtain the simple integral expression
\begin{align}
\hbox{vol}_5(P(\alpha))=
\frac{1}{4}\,\int\limits_{\alpha}^{2\pi/5}\hbox{vol}_3([5,3,\beta(t)])\,dt+\frac{\zeta(3)}{3200}
	\label{eq:differential}
\end{align}
with a compact tetrahedron $[5,3,\beta(t)]$ whose angle parameter $\beta(t)\in\,]0,\pi/2[$ is given by
\begin{align}
\beta(t)=\arctan{\sqrt{2-\cot^2t}}\,\,.
\label{eq:beta}
\end{align}
Put
\begin{align}
\theta(t)=\arctan\frac{\sqrt{1-4\sin^2\frac{\pi}{5}\sin^2\beta(t)}}{2\cos\frac{\pi}{5}\cos \beta(t)}\in\,]0,\frac{\pi}{2}[\,\,.
\label{eq:theta}
\end{align}
Then, the volume of the 3-dimensional orthoscheme face $[5,3,\beta(t)]$ as given by Lobachevsky's formula
(see \cite{Kel12}, (67), for example) equals
\begin{align}
\label{eq:vol3}
\hbox{vol}&_3(\,[5,3,\beta(t)]\,)=\frac{1}{4}\,\big\{\,{\russianL}_2(\frac{\pi}{5}+\theta(t))-
{\russianL}_2(\frac{\pi}{5}-\theta(t))- {\russianL}_2\left(\frac{\pi}{6}+\theta(t)
\right)+\\
\nonumber
&+{\russianL}_2\left(\frac{\pi}{6}{-}\theta(t)\right)+{\russianL}_2(\beta(t){+}\theta(t))-{\russianL}_2(\beta(t){-}\theta(t))+2{\russianL}_2\left(\frac{\pi}{2}{-}\theta(t)\right)\big\}\,\,,
\end{align}
where
\begin{align}
{\russianL}_2(\omega)=
{1\over2}\,\sum_{r=1}^{\infty}{\sin(2r\omega)\over r^2}
=-\int\limits_0^{\omega}\,\log\mid2\sin t\mid dt\,\,,\,\,\omega\in\mathbb R\,\,,
	\label{eq:loba2}
\end{align}
is Lobachevsky's function (in a slightly modified way).

The numerical approximation of the volumes of $P_0$ and $P_2$ can now be performed 
by implementing the data (\ref{eq:beta}), (\ref{eq:theta}) and
\eqref{eq:vol3} into the expression \eqref{eq:differential}. We obtain,
using the  functions \texttt{intnum} and \texttt{polylog} in \textsc{Pari/GP}, 
that the three volumes of $P_0$, $P_1$ and $P_2$ (in increasing order) are clearly less than
$2 \cdot 10^{-3}$. Since  the groups $\Delta_i$ ($i=0,1,2$) are known to
be arithmetic, it follows then from Proposition
\ref{prop:smallest-three} that their subgroups of index two $\Delta_i^+$
must coincide with the $\Gamma_i'$. This concludes the proof of Theorem
\ref{thm:Cox}.

\section{Remarks on the identification of volumes}
\label{sec:rmks-vol}

Although in the proof of Theorem \ref{thm:Cox} it suffices to use the rough estimate
$\vol(P_2) < 2 \cdot 10^{-3}$, the numerical approximations are much more precise. 
Namely, the equality $\vol(\Gamma'_i\bs \Hy^5) =
\vol(\Delta^+_i\bs \Hy^5)$, proved by Theorem \ref{thm:Cox}, yields for $i=0,2$:
\begin{align}
	\begin{split}
 \frac{9 \sqrt{5}^{15} \sqrt{11}^5}{2^{14} \pi^{15}} \zeta_{k_0}(2)
	\label{eq:eq-vol}
 \zeta_{k_0}(4) L_{\ell_0/k_0}(3) &= 2 \vol_5(P(\pi/3));\\
 \frac{9 \sqrt{5}^{15}}{2^3 \pi^{15}} \zeta_{k_0}(2)
 \zeta_{k_0}(4) L_{\ell_2/k_0}(3) &= 2 \vol_5(P(\pi/5)).
	\end{split}
\end{align}
Using \textsc{Pari/GP}, a computer checks within seconds that both sides
of each equation coincide up to 50 digits (the right hand side being
computed from \eqref{eq:differential} like in last step of \S
\ref{sec:proof-2}).

The equalities \eqref{eq:eq-vol} have also some arithmetic interest, due
the presence on the left hand side of the special value $L_{\ell/k_0}(3)$
(with $\ell = \ell_0$ or $\ell_2$). Since $k_0$ is totally real, it
follows from the Klingen-Siegel theorem (see \cite{Kling62}; cf. also
\cite[App. C]{MohamSG}) that
$\zeta_{k_0}(2)
\zeta_{k_0}(4)$ is up to a rational given by some power of $\pi$ divided
by $\sqrt{\D_{k_0}} = \sqrt{5}$. Thus, from \eqref{eq:eq-vol} the nontrivial part
$L_{\ell/k_0}(3)$ of $\vol(\Gamma_i' \bs \Hy^5)$ can be expressed by a
sum of integrals of Lobachevsky's functions. A related but much more
significant idea is the possibility, predicted by Zagier's conjecture, 
to express $L_{\ell/k_0}(3)$ as a sum of trilogarithms evaluated at 
integers of $k_0$. We refer to
\cite{ZagGan00} for more information on this subject.


\bibliographystyle{amsplain}
\bibliography{vol5}

\end{document}